\documentclass[11pt,leqno]{article}
\usepackage{amsthm,amsfonts,amssymb,amsmath,oldgerm}
\usepackage{epsfig}
\numberwithin{equation}{section}
\usepackage[thinlines]{easybmat}

\newcommand\br{\begin{remark}}
\newcommand\er{\end{remark}}
\newcommand\bp{\begin{pmatrix}}
\newcommand\ep{\end{pmatrix}}
\newcommand\be{\begin{equation}}
\newcommand\ee{\end{equation}}
\newcommand\ba{\begin{equation}\begin{aligned}}
\newcommand\ea{\end{aligned}\end{equation}}

\newcommand{\bap}{\begin{app}}
\newcommand{\eap}{\end{app}}
\newcommand{\begs}{\begin{exams}}
\newcommand{\eegs}{\end{exams}}
\newcommand{\beg}{\begin{example}}
\newcommand{\eeg}{\end{exaplem}}
\newcommand{\bpr}{\begin{proposition}}
\newcommand{\epr}{\end{proposition}}
\newcommand{\bt}{\begin{theorem}}
\newcommand{\et}{\end{theorem}}
\newcommand{\bc}{\begin{corollary}}
\newcommand{\ec}{\end{corollary}}
\newcommand{\bl}{\begin{lemma}}
\newcommand{\el}{\end{lemma}}
\newcommand{\bd}{\begin{definition}}
\newcommand{\ed}{\end{definition}}
\newcommand{\brs}{\begin{remarks}}
\newcommand{\ers}{\end{remarks}}

\newtheorem{theo}{Theorem}[section]

\newtheorem{exams}[theo]{Examples}

\numberwithin{equation}{section}

\newtheorem{theorem}{Theorem}[section]
\newtheorem{proposition}[theorem]{Proposition}
\newtheorem{corollary}[theorem]{Corollary}
\newtheorem{lemma}[theorem]{Lemma}
\newtheorem{definition}[theorem]{Definition}

\newtheorem{example}[theorem]{Example}
\newtheorem{remark}[theorem]{Remark}

\newcommand{\beq}{\begin{equation}}
\newcommand{\eeq}{\end{equation}}
\newcommand{\bs}{\begin{split}}
\newcommand{\es}{\end{split}}

\title{Stability of ZND detonations for Majda's model}

\author{\sc \small
Soyeun Jung \thanks{Indiana University, Bloomington, IN 47405;
soyjung@indiana.edu}
~~~~~
Jinghua Yao \thanks{Indiana University, Bloomington, IN 47405;
yaoj@indiana.edu :
Research of S.J. and J.Y. was partially supported under NSF grants
number DMS-0070765 and DMS-0300487. Thanks to Kevin Zumbrun for
suggesting the problem and for helpful discussions.
 }}

\begin{document}

\maketitle

\begin{abstract}
We evaluate by direct calculation the Lopatinski determinant
for ZND detonations in Majda's model for reacting flow, and show
that on the nonstable (nonnegative real part) complex half-plane it
has a single zero at the origin of multiplicity one, implying
stability. Together with results of Zumbrun on the inviscid limit,
this recovers the result of RoqueJoffre--Vila that viscous
detonations of Majda's model also are stable for sufficiently small
viscosity, for any fixed detonation strength, heat release, and rate
of reaction.

\end{abstract}

\section{Introduction}\label{S.1}
In this note, we verify by explicit computation the spectral
stability in the sense of Erpenbeck \cite{Er} of strong ZND
detonation wave solutions of Majda's model for reactive gas dynamics
\cite{M} with a step-type ignition function.

Consider the inviscid Majda's model
\begin{equation}\label{1.1}
\begin{array}{rll}
(u+qz)_t+\displaystyle\left(\frac{u^2}{2}\right)_x & = & 0, \\
z_t+k\varphi(u)z & = & 0\\
\end{array}
\end{equation}
an analog of the Zeldovich-Von Neumann-Doering (ZND) equations for
reactive gas dynamics, where $z, \varphi \in \mathbb{R}$, $u \geq 0$, $q \ge
0$ and $k>0$. Here, $u$ is a lumped variable modeling the
gas-dynamical quantities of density, momentum, and energy, $z$ is
mass fraction of reactant, $q\ge 0$ a coefficient of heat release of
the reaction, $k>0$ reaction rate, and $\varphi(u)$ is a simple
step-type ``ignition function" that is assumed to be zero below a
certain value $u_i>0$ and one above.

A strong detonation wave of \eqref{1.1} is a traveling-wave solution
\begin{equation}\label{1.2}
(u,z)(x,t)  =  (\overline{u},\overline{z})(x-st), \quad
\lim_{\xi\rightarrow\pm\infty}(\overline{u},\overline{z})(\xi)=(u_\pm,
z_\pm)
 \end{equation}
in the weak, or distributional, sense, smooth except at a single
shock discontinuity at (without loss of generality) $x=0$, known as
a ``Neumann shock", where $u$ jumps from $u_*:=\bar u(0^-)$ to $\bar
u(0^+)$ as $x$ crosses zero from left to right, and satisfying
\begin{equation}\label{1.3}
    z_-=0,  \, z_+=1,
\quad u_-> u_i >u_+.
\end{equation}
and
\begin{equation}\label{Lax}
u_->s>u_+ \geq 0.
\end{equation}
Here, analogy of the model for gas dynamics is only for $u_+ \geq 0$, the ``physical" range defined by Majda.
Computing the Rankine--Hugoniot conditions at the shock at $x=0$, we
find that
\begin{equation}\label{RH}
\begin{aligned}
s(\bar u(0^+)-\bar u(0^-))&=\frac{\bar u(0^+)^2}{2}- \frac{\bar u(0^-)^2}{2},\\
\bar z(0^-)&=\bar z(0^+),
\end{aligned}
\end{equation}
which satisfies
\begin{equation}
u_+^2-2su_+=u_*^2-2su_* \quad \text{or} \quad s=\frac{u_+ +u_*}{2},
\end{equation}
letting $u_-$, $q>0$, $k>0$ vary.
From \eqref{Lax} and the
assumption that $(\bar u,\bar z)$ converges as $x\to \pm \infty$, we
find further by consideration of the traveling-wave ODE (see Section
\ref{S.3}) that
\begin{equation}\label{const}
(\bar u, \bar z)(x)\equiv (u_+,z_+) \; \hbox{\rm for } \; x> 0
\end{equation}
and also $u_-= s+\sqrt{s^2-2qs+u_+^2-2su_+}$ ,
 so that
\begin{equation}\label{range}
s <  u_- < u_*, \quad 0 < q < \frac{(u_*-s)^2}{2s},
\end{equation}
with reaction rate varying in the infinite range $0<k<+\infty$.

That is, we have the standard picture of a strong detonation wave as
a shock advancing to the right into a quiescent (i.e., nonreacting)
constant state with reactant mass fraction $z=1$, raising $u$ above
ignition level $u_i$, followed by a smooth ``reaction tail'' in
which combustion (reaction) occurs, in which $z$ decays
exponentially to value $z_-=0$ and $u$ to $s < u_-<u_*$ as $x\to
-\infty$. See \cite{LyZ1,LyZ2,JLW,Z1,Z2} for further details.

As shown by Erpenbeck \cite{Er}, spectral stability of such waves,
defined as nonexistence of normal modes $e^{\lambda t}w(x)$ with
$\Re \lambda \ge 0$ and $\lambda\ne 0$, may be determined by
examination of a certain Lopatinski determinant $D_{ZND}(\lambda)$
whose zeros on $\Re \lambda \ge 0$ correspond to eigenvalues
$\lambda$. Precisely, spectral stability corresponds to the
condition
\begin{equation}\tag{D}
D_{ZND}(\cdot) \; \hbox{\rm has a single zero on}\; \{\Re \lambda
\ge 0\} \; \hbox{\rm occurring at} \; \lambda=0.
\end{equation}

The Lopatinski condition (D) has been much studied numerically in
both one and multi dimensions, for the full equations of reactive
gas dynamics; see, for example, \cite{KS} and references therein.
However, so far as we know, it has up to now not been verified
analytically for any case. Here, we show by direction calculation
that (D) holds for detonations of the inviscid Majda model with
step-type ignition function.

Since linear or spectral stability concerns only behavior near
values of the profile $(\bar u,\bar z)$, it is clear that this
result extends also to the slightly more general class of (possibly
smooth) ignition functions identified by Roquejoffre and Vila
\cite{RV} that $\phi(u)\equiv 0$ for $u<u_i$ and $\phi(u)\equiv 1$
for $u>u^i$, provided that $ u_->u^i>u_i>u_+. $ That is, we require
in our analysis only that
\begin{equation}\label{weakprop}
\phi(u)=0 \; \hbox{\rm  near  $u_+$  and $\phi(u)=1$ near
$[u_-,u_*]$}.
\end{equation}
A very interesting open problem is whether a corresponding stability
result holds for more general ignition functions satisfying only
\begin{equation}\label{gen}
\phi(u)=0 \; \hbox{\rm  for $u<u_i$  and $\phi(u)>0$ for  $u>u_i$}.
\end{equation}

As discussed in Section \ref{S.6}, these results have implications
also for spectral and nonlinear stability of viscous detonation
waves in the ZND limit as viscosity goes to zero.

\section{Profile solutions}\label{S.3}


Observing that
\begin{equation*}
    \partial_t \overline{u}(x-st)  =  -s\overline{u} ^{\prime},
\quad
\partial_x\overline{u}(x-st)=\overline{u} ^{\prime}
\quad
\partial_t\overline{z}(x-st)=-s\overline{z}^{\prime},
\quad
\partial_x\overline{z}(x-st)=\overline{z}^{\prime},
\end{equation*}
we obtain for a detonation solution \eqref{1.2}--\eqref{RH} the
profile equations
\begin{equation}\label{2.1}
\begin{aligned}
 \displaystyle -s(\overline{u}+q\overline{z})^{\prime}+ \left(\frac {\overline{u}^2}{2}\right)^{\prime} &= 0  \\
  -s\overline{z}'+k\varphi(\overline{u})\overline{z}&=0,
\end{aligned}
\end{equation}
valid on the regions of smoothness $x>0$ and $x<0$. Linearizing
\eqref{2.1} about the equilibrium solution $(u,z)\equiv (u_+,1)$, we
readily find from $s>u_+$ that this is a repelling equilibrium, and
so the only smooth solution of \eqref{2.1} converging to this
equilibrium is the constant solution \eqref{const}.

%
Integrating \eqref{2.1} from $-\infty$ to $+\infty$, we obtain
\begin{equation*}\label{}
    s(u_+ +qz_+)-\frac{{u_+}^2}{2} = s(u_- +qz_-)-\frac{{u_-}^2}{2}.
\end{equation*}
Substituting $z_+=1$ and $z_-=0$, we find that
\begin{equation}\label{3.1}
    2s(u_+ + q)-u_+^2=2su_--u_-^2
\end{equation}

Now, integrating \eqref{2.1} from $-\infty$ to $\xi$, we obtain
$$
  -s(\overline{u}+q\overline{z})+s(\overline{u}_-+q\overline{z}_-)=\displaystyle-\frac{\overline{u}^2}{2}+\frac{\overline{u}_-^2}{2},
$$
from which, by \eqref{1.3} and \eqref{3.1}, we obtain
\begin{equation}\label{ueq}
\overline{u}(\xi)=s\pm\sqrt{s^2-2qs(1-\overline{z}(\xi))+u_+^2-2su_+}.
\end{equation}
Recalling $u_*$ is the left-limit of $\overline{u}(\xi)$, $u_* > s$ and $\overline{z}(0)=1$, we choose the sign $+$, yielding in particular $u_-=
s+\sqrt{s^2-2qs+u_+^2-2su_+}$. Since $\varphi(\overline{u})=1$ for $\xi<0$ in
\eqref{2.1}, $\displaystyle \overline{z}^{\prime}=(k/s)\overline{z}$, hence for
$\xi<0$,
\begin{equation}\label{3.3}
  \overline{z}(\xi)=\overline{z}(0)e^{\frac{k}{s}\xi}=e^{\frac{k}{s}\xi},
\end{equation}
and
\begin{equation}\label{3.4}
  \overline{u}(\xi)=s+ \sqrt{s^2-2qs(1-e^{\frac{k}{s}\xi})+u_+^2-2su_+}.
\end{equation}

\section{Eigenvalue equations}\label{S.4}

Let $(\tilde{u},\tilde{z})$ be a solution of \eqref{1.1} different
from $(\overline{u},\overline{z})$. Subtracting
$(\tilde{u},\tilde{z})$ from $(\overline{u},\overline{z})$, we
obtain for the perturbation variable
\begin{equation*}
(u,z):=(\tilde{u},\tilde{z})-(\overline{u},\overline{z})
\end{equation*}
the perturbation equations
\begin{eqnarray*}
(u+qz)_t+\left(\frac{\tilde{u}^2}{2}-\frac{\overline{u}^2}{2}\right)_{\xi}&=&0  \\
z_t+k(\varphi(\tilde{u})\tilde{z}-\varphi(\overline{u})\overline{z})&=&0.
\end{eqnarray*}
Taylor expanding and dropping $O(|(u,z)|^2)$ terms, we obtain the
linearized equations
$$
\begin{aligned}
  u_t-su_{\xi}+(\overline{u}u)_{\xi} &= qk(d\varphi(\overline{u})\overline{z}u+\varphi(\overline{u})z) \\
  z_t-sz_{\xi} &= -k(d\varphi(\overline{u})\overline{z}u+\varphi(\overline{u})z),
\end{aligned}
$$
and, finally, the linearized eigenvalue equations
\begin{eqnarray*}
  \lambda u-su'+(\overline{u}u)' &=& qk(d\varphi(\overline{u})\overline{z}u+\varphi(\overline{u})z) \\
  \lambda z-sz' &=& -k(d\varphi(\overline{u})\overline{z}u+\varphi(\overline{u})z),
\end{eqnarray*}
where $\prime$ denotes $\partial_{\xi}$.

Expanding and rearranging, we have
$$
\begin{aligned}
  (\overline{u}-s)u'+(\overline{u}-s)'u &=(qkd\varphi(\overline{u})\overline{z}-\lambda)u+qk\varphi(\overline{u})z  \\
  -sz' &= -kd\varphi(\overline{u})\overline{z}u+(-k\varphi(\overline{u})-\lambda)z
\end{aligned}
$$
which gives the matrix equation
\begin{eqnarray*}
  \left\{\left(
    \begin{array}{cc}
      \overline{u}-s & 0 \\
      0 & -s \\
    \end{array}
  \right)
  \left(
    \begin{array}{c}
      u \\
      z \\
    \end{array}
  \right)\right\}'
   &=& \left(
         \begin{array}{cc}
           qkd\varphi(\overline{u})\overline{z}-\lambda & qk\varphi(\overline{u}) \\
           -kd\varphi(\overline{u})\overline{z} & -k\varphi(\overline{u})-\lambda \\
         \end{array}
       \right)
   \left(
     \begin{array}{c}
       u \\
       z \\
     \end{array}
   \right)
    \\
    &=&\left\{\left(
      \begin{array}{cc}
        qkd\varphi(\overline{u})\overline{z} & qk\varphi(\overline{u}) \\
        -kd\varphi(\overline{u})\overline{z} & -k\varphi(\overline{u}) \\
      \end{array}
    \right)-\lambda I\right\}
    \left(
      \begin{array}{c}
        u \\
        z \\
      \end{array}
    \right).
\end{eqnarray*}

Setting
\begin{equation}\label{4.3}
A=\left(
    \begin{array}{cc}
      \overline{u}-s & 0 \\
      0 & -s \\
    \end{array}
  \right),
\;
  W=\left(
    \begin{array}{c}
      u \\
      z \\
    \end{array}
  \right),
\;
  E=\left(
      \begin{array}{cc}
        qkd\varphi(\overline{u})\overline{z} & qk\varphi(\overline{u}) \\
        -kd\varphi(\overline{u})\overline{z} & -k\varphi(\overline{u}) \\
      \end{array}
    \right),
\;
    Z=AW,
\end{equation}
we may write the eigenvalue equations as the first-order ODE system
\begin{equation}\label{Gode}
  Z'=  (E-\lambda I)W
   = (E-\lambda I)A^{-1}Z
   = GZ,
\end{equation}
where $ G=(E-\lambda I)A^{-1}. $

\section{The Lopatinski determinant} \label{S.5}

Following \cite{JLW,Z2}, define on $\Re \lambda \ge 0$ the
Lopatinski determinant
\begin{equation}\label{5.1}
    \displaystyle D_{ZND}(\lambda):=\det (Z^-(\lambda,0),
\; \lambda[\overline{W}]-[A \overline{W}^{\prime} ]),
\end{equation}
where $[h]:=h(0^+)-h(0^-)$ and $Z^-(\lambda, \xi)$ is a bounded
exponentially decaying solution of \eqref{Gode}, analytic in
$\lambda$ and tangent as $\xi \rightarrow -\infty$ to the subspace
of exponentially decaying solutions of the limiting,
constant-coefficient equations $Z^{\prime}=G_-Z.$ By standard
asymptotic ODE theory (the ``gap lemma'' \cite{GZ}), $Z^-$ is
uniquely determined up to a nonvanishing analytic factor not
affecting stability.

\begin{lemma}
The ZND Lopatinski determinant is given (up to nonvanishing analytic
factor) by
\begin{equation}\label{ZNDdet}
D_{ZND}= \big( (u_*-u_+)\lambda+(u_*-u_+ -q-qk\Psi)k \big)
\Big(\frac{\lambda}{k+\lambda}\Big) ,
\end{equation}
\begin{equation}\label{5.4}
\begin{split}
\Psi :& =\int_{-\infty}^{0}e^{-\int_{y}^{0} P(s)
ds}\frac{e^{\frac{(k+\lambda)}{s}y}}{\sqrt{s^2-2qs(1-e^{(k/s)y})+u_+^2-2su_+}}dy, \\
P(\xi):& =\frac{\lambda}{\sqrt{s^2-2qs(1-e^{(k/s)\xi})+u_+^2-2su_-}}.
\end{split}
\end{equation}
\end{lemma}

\begin{proof}
Combining \eqref{3.3}, \eqref{3.4} and \eqref{4.3}, we obtain
$$
\begin{aligned}
  \lambda[\overline{W}]-[A(\overline{W}^{\prime})]
   & = \displaystyle \lambda \displaystyle \left[ \left(
                 \begin{array}{c}
                   \overline{u} \\
                   \overline{z} \\
                 \end{array}
               \right)\right]
               +(A \overline{W}^{\prime})(0^-) \\
   &= \displaystyle \lambda \left(
                 \begin{array}{c}
                   u_+ - u_* \\
                   0 \\
                 \end{array}
               \right)
               +\left(
                  \begin{array}{cc}
                    u_*-s & 0 \\
                    0 & -s \\
                  \end{array}
                \right)\left(
                         \begin{array}{c}
                          (qk)/(u_*-s) \\
                         k/s \\
                         \end{array}
                       \right)\\
   &= \displaystyle \left(
          \begin{array}{c}
            \lambda u_+ - \lambda u_* + qk \\
            -k \\
          \end{array}
        \right),
\end{aligned}
$$
where we are using $ \overline{W}^{\prime}=\Big(
\displaystyle\frac{qke^{(k/s)\xi}}{\sqrt{s^2-2qs(1-e^{(k/s)\xi})+u_+^2-2su_+}} , \frac{k}{s}e^{(k/s)\xi}\Big)^T$.
Computing
$$
\begin{aligned}
                G = (E-\lambda I)A^{-1} &= \left(
                       \begin{array}{cc}
                         -\lambda & qk \\
                         0 & -k-\lambda \\
                       \end{array}
                     \right)
                     \left(
                       \begin{array}{cc}
                         \overline{u}-s & 0 \\
                         0 & -s \\
                       \end{array}
                     \right)^{-1}
                 = \left(
                       \begin{array}{cc}
                         \displaystyle \frac{-\lambda}{\overline{u}-s} & \displaystyle \frac{-qk}{s} \\
                         0 & \displaystyle \frac{k+\lambda}{s} \\
                       \end{array}
                     \right)
\end{aligned}
$$
and setting $Z=\left(
             \begin{array}{c}
               Z_1 \\
               Z_2 \\
             \end{array}
           \right)$,
we have
$$
    \left(
      \begin{array}{c}
        Z_1(\lambda, \xi) \\
        Z_2(\lambda, \xi) \\
      \end{array}
    \right)^{\prime}=\left(
                       \begin{array}{cc}
                         \displaystyle \frac{-\lambda}{\overline{u}-s} & \displaystyle \frac{-qk}{s} \\
                         0 & \displaystyle \frac{k+\lambda}{s} \\
                       \end{array}
                     \right)\left(
                              \begin{array}{c}
                               Z_1(\lambda, \xi) \\
                               Z_2(\lambda, \xi) \\
                              \end{array}
                            \right),
$$
or
\begin{equation*}\label{}\begin{array}{rll}
Z_1^{\prime}&=&\displaystyle -\frac{\lambda}{\overline{u}-s}Z_1-\frac{qk}{s}Z_2, \\\\
Z_2^{\prime} &=&\displaystyle \frac{k+\lambda}{s}Z_2.
  \end{array}
\end{equation*}
Solving, we obtain
\begin{equation*}\label{}\begin{array}{rll}
Z_1^{\prime} &=& -\displaystyle \frac{\lambda}{\sqrt{s^2-2qs(1-e^{(k/s)\xi})+u_+^2-2su_+}}Z_1-\frac{qk}{s}e^{\frac{(k+\lambda)}{s}\xi}, \\\\
Z_2 &=& e^{\frac{(k+\lambda)}{s}\xi}.
  \end{array}
\end{equation*}

Setting now $\displaystyle
P(\xi):=\frac{\lambda}{\sqrt{s^2-2qs(1-e^{(k/s)\xi})+u_+^2-2su_+}}$ and $\displaystyle
Q(\xi):=-\frac{qk}{s}e^{\frac{(k+\lambda)}{s}\xi}$, we have
\begin{equation}\label{5.2}
\begin{split}
\displaystyle\int P(\xi)d\xi
& =-\frac{2\lambda s}{k\sqrt{s^2-2qs+u_+^2-2su_+}} \\
& \qquad \times \ln\left(\frac{\sqrt{s^2-2qs(1-e^{(k/s)\xi})+u_+^2-2su_+}+\sqrt{s^2-2qs+u_+^2-2su_+}}{\sqrt{2qse^{(k/s)\xi}}}\right),
\end{split}
\end{equation}
which gives $\Big|\displaystyle
e^{-\int_{-\infty}^{0}P(\xi)d\xi}\Big|\leq1$ for $\Re \lambda
\geq0$, and hence, integrating  by parts,
\begin{equation}\label{5.3}
\begin{aligned}
\displaystyle
Z_1(\lambda,0)&= \displaystyle e^{-\int_{-\infty}^{0}P(s)ds}\left(\int_{-\infty}^{0} e^{\int_{-\infty}^{y}P(s)ds}Q(y)dy\right)\\
    &= \displaystyle \int_{-\infty}^{0} e^{-\int_{y}^{0}P(s)ds}Q(y)dy\\
    &= \displaystyle -e^{-\int_{y}^{0}P(s)ds} \frac{qk}{k+\lambda}e^{\frac{(k+\lambda)}{s}y}\bigg|_{-\infty}^{0}+\int_{-\infty}^{0} e^{-\int_{y}^{0}P(s)ds}P(y)\frac{qk}{k+\lambda}e^{\frac{(k+\lambda)}{s}y}dy\\
    &= \displaystyle -\frac{qk}{k+\lambda}\left(1-\lambda \Psi\right),
    \end{aligned}
\end{equation}
where $\Psi$ is as in \eqref{5.4}. By \eqref{5.1}, \eqref{5.2}, and
\eqref{5.3},
\begin{equation}\label{5.5}
\begin{aligned}
    \displaystyle D_{ZND}(\lambda)&=\det (Z^-(\lambda,0),
\;
\lambda[\overline{W}]+A(\overline{W}^{\prime})(0^-)) \\
   &= \det\left(
             \begin{array}{cc}
               Z_1(\lambda,0) & \lambda u_+ - \lambda u_* + qk \\
               Z_2(\lambda,0) & -k \\
             \end{array}
           \right)\\
   &= -kZ_1 + \lambda u_* - \lambda u_+ - qk \\
   &= \displaystyle \frac{qk^2(1-\lambda\Psi)+(\lambda u_* - \lambda u_+ - qk )(k+\lambda)}{k+\lambda}\\
   &= \displaystyle \big( (u_*-u_+)\lambda+(u_*-u_+ - q-qk\Psi)k \big)
\displaystyle \left(\frac{\lambda}{k+\lambda}\right).
   \end{aligned}
\end{equation}
\end{proof}

\section{Verification of spectral stability}

\begin{theorem}\label{main}
For a step-type ignition function, or, more generally, any ignition
function satisfying \eqref{weakprop}, $D_{ZND}$ has a single zero of
multiplicity one on $\{\Re \lambda\ge 0\}$, located at $\lambda=0$;
that is, the Lopatiski condition (D) is satisfied for all ZND
detonations of Majda's model, independent of the choice of $q > 0$,
$k>0$, or $u_\pm$.
\end{theorem}

\begin{proof}
By \eqref{5.4}, we have for $\Re\lambda \ge 0$
\begin{equation*}
\begin{aligned}
|\Psi|
&\le  \displaystyle \int_{-\infty}^{0} e^{-\int_{y}^{0}\Re P(s)ds}\frac{e^{\frac{(k+\Re \lambda)}{s}y}}{\sqrt{s^2-2qs(1-e^{(k/s)y})+u_+^2-2su_+}}dy \\
&\leq  \displaystyle \int_{-\infty}^{0}\frac{e^{(k/s)y}}{\sqrt{s^2-2qs(1-e^{(k/s)y})+u_+^2-2su_+}}dy \\
&=  \displaystyle \frac{1}{qk} \int_{-\infty}^{0} \frac{d}{dy}
\Big( {\sqrt{s^2-2qs(1-e^{(k/s)y})+u_+^2-2su_+}} \Big) dy \\
   &= \displaystyle \frac{1}{qk}\left( \sqrt{s^2+u_+^2-2su_+}-  \sqrt{s^2-2qs+u_+^2 -2su_+} \right) \\
& = \displaystyle \frac{1}{qk} \left( (s-u_+) - \sqrt{s^2-2qs+u_+^2 -2su_+} \right)
\end{aligned}
\end{equation*}
giving
\begin{equation}
\begin{split}
\Re (u_*-u_+ -q-qk\Psi)
& \geq u_*-u_+ -q-qk|\Psi| \\
& \geq u_*-u_+ - q - s+u_+ +  \sqrt{s^2-2qs+u_+^2-2su_+} \\
& = u_* -s - q +\sqrt{s^2-2qs+u_+^2-2su_+}\\
& = s - u_+ - q +\sqrt{s^2-2qs+u_+^2-2su_+}
\end{split}
\end{equation}
By $0 \leq u_+ < s$ and $2qs < (s-u_+)^2 \leq s^2$,
\begin{equation}
\begin{split}
\Re (u_*-u_+ -q-qk\Psi)
& > \sqrt{2qs}-q+\sqrt{s^2-2qs+u_+^2-2su_+} \\
& > \sqrt{4q^2}-q+\sqrt{s^2-2qs+u_+^2-2su_+} \\
& > 0.
\end{split}
\end{equation}
In particular,
\begin{equation}\label{5.6}
 u_*-u_+ -q-qk\Psi \ne 0 \quad \hbox{\rm for } \quad \Re \lambda \ge 0.
\end{equation}
Combining \eqref{5.5} and \eqref{5.6}, we obtain the result.
\end{proof}

\section{Viscous stability and a result of Roquejoffre--Vila}\label{S.6}
%

Our results have implications also for stability of ``viscous''
detonation waves, i.e., smooth analogs of traveling waves
\eqref{1.2} satisfying the ``viscous'' or parabolic regularization
of \eqref{1.1}:
\begin{equation}\label{reg}
\begin{array}{rll}
(u+qz)_t+\displaystyle\left(\frac{u^2}{2}\right)_x & = &
\epsilon qz_{xx}+ \epsilon u_{xx}, \\
z_t+k\varphi(u)z & = & \epsilon z_{xx},\\
\end{array}
\end{equation}
with $\epsilon >0$. It is shown in \cite{Z2} that stability of
viscous detonation waves in the ZND limit as $\epsilon \to 0$ with
other parameters held fixed is equivalent to Lopatinski stability of
the limiting ZND detonation \eqref{1.2} together with Evans
(equivalently, linearized) stability of the viscous Burgers shock
corresponding to the Neumann shock contained in the ZND detonation
profile, as is well-know to hold by properties of scalar
traveling-waves (see, e.g., \cite{Sa}).

Thus, our results together with those of \cite{Z2} yield the result
for \eqref{reg} of spectral stability of viscous detonation waves in
the ZND limit, similar to an earlier result of Roquejoffre--Vila
\cite{RV} for the corresponding equations with regularization
$\epsilon u_{xx}$ in the $u$ equation alone and applying to the same
class of ignition functions \eqref{weakprop}.
A very interesting open problem would be to extend our results to
the more general class of ignition functions \eqref{gen}, which
would give new information for the viscous stability problem as
well.

We remark that results of \cite{LRTZ} (for Majda's model) and
\cite{TZ} (for the physical reactive Navier--Stokes equations) show
that spectral stability of viscous detonation waves in the (Evans
function) sense of \cite{Z2} implies linearized and nonlinear
orbital stability, hence stability of strong viscous detonation
waves reduces for viscosity sufficiently small to a study of
spectral ZND stability as carried out here.



\begin{thebibliography}{GMWZ7}



%

\bibitem [Er]{Er} J. J. Erpenbeck,
{\it Stability of idealized one-reaction detonations,} Phys. Fluids
7 (1964).

%


\bibitem[GZ]{GZ}  R. Gardner and K. Zumbrun,
{\it The Gap Lemma and geometric criteria for instability of viscous
shock profiles}. Comm. Pure Appl.  Math. 51 (1998), no. 7, 797--855.

\bibitem[JLW]{JLW} H.K. Jenssen, G. Lyng, and M. Williams.
{\it Equivalence of low-frequency stability conditions for
  multidimensional detonations in three models of combustion,}
Indiana Univ. Math. J. 54 (2005) 1--64.

\bibitem[KS]{KS} A.R. Kasimov and D.S. Stewart,
{\it Spinning instability of gaseous detonations.} J. Fluid Mech.
466 (2002), 179--203.

\bibitem[LyZ1]{LyZ1} G. Lyng and K. Zumbrun,
{\it One-dimensional stability of viscous strong detonation waves,}
Arch. Ration. Mech. Anal.  173  (2004),  no. 2, 213--277.

\bibitem[LyZ2]{LyZ2} G. Lyng and K. Zumbrun,
\emph{A stability index for detonation waves in Majda's model for
reacting flow}, Physica D, {194} (2004), 1--29.

\bibitem[LRTZ]{LRTZ} G. Lyng, M. Raoofi, B. Texier, and K. Zumbrun,
\emph{Pointwise Green Function Bounds and stability of combustion
waves}, J. Differential Equations  233  (2007) 654--698.

\bibitem[M]{M} A. Majda,
{\it A qualitative model for dynamic combustion}, SIAM J. Appl.
Math., 41 (1981), 70--91.

\bibitem [RV]{RV} J.-M. Roquejoffre and J.-P. Vila,
{\it Stability of ZND detonation waves in the Majda combustion
model,} Asymptot. Anal. 18 (1998), no. 3-4, 329--348.

\bibitem[Sa]{Sa} D. Sattinger,
{\it On the stability of waves of nonlinear parabolic systems}. Adv.
Math. 22 (1976) 312--355.

\bibitem[TZ]{TZ} B. Texier and K. Zumbrun, {\it Transition to longitudinal instability of detonation waves is generically associated with Hopf bifurcation to time-periodic galloping solutions},
preprint (2008).

\bibitem[Z1]{Z1} K. Zumbrun,
{\it Multidimensional stability of planar viscous shock waves,}
Advances in the theory of shock waves, 307--516, Progr. Nonlinear
Differential Equations Appl., 47, Birkh\"auser Boston, Boston, MA,
2001.

\bibitem[Z2]{Z2} K. Zumbrun,
{\it Stability of viscous detonation waves in the ZND limit},
preprint (2009).



\end{thebibliography}
\end{document}